\documentclass[12pt,reqno]{amsart}

\usepackage{amsmath}
\usepackage{amsfonts}
\usepackage{amssymb}
\usepackage{verbatim}
\usepackage[left=2.9cm,right=2.9cm,top=3cm,bottom=3cm]{geometry}
\usepackage{enumerate}

\usepackage{graphicx}

\newcommand{\mP}{{\mathbb P}}

\newtheorem{theorem}{Theorem}
\newtheorem{lemma}{Lemma}
\newtheorem{rem}{Remark}

\newtheorem{prop}{Proposition}

\newtheorem{example}{Example}
\newtheorem{definition}{Definition}
\newtheorem{corollary}{Corollary}

\title[Stochastic precedence and minima among dependent variables]{Stochastic precedence and minima among dependent variables}

\author{Emilio De Santis}
\address{University of Rome La Sapienza, Department of Mathematics
Piazzale Aldo Moro, 5, I-00185, Rome, Italy}
\email{desantis@mat.uniroma1.it}

\author{Yaakov Malinovsky}
\address{Department of Mathematics and Statistics
University of Maryland, Baltimore County
1000 Hilltop Circle
Baltimore, MD 21250}
\email{yaakovm@umbc.edu}

\author{Fabio Spizzichino}
\address{University of Rome La Sapienza
Piazzale Aldo Moro, 5, I-00185, Rome, Italy}
\email{fabio.spizzichino@fondazione.uniroma1.it}

\begin{document}

\begin{abstract}
The notion of stochastic precedence between two random variables emerges as a relevant concept in several fields of applied probability.
When one consider a vector of  random variables $X_1,...,X_n$, this notion has a preeminent role in the analysis of minima of the type $\min_{j \in A} X_j$
for $A \subset \{1, \ldots n\}$.
In such an analysis, however, several apparently controversial aspects can arise (among which phenomena of ``non-transitivity").
        Here we concentrate attention on vectors of non-negative random variables with absolutely continuous  joint distributions, in which a case the set of the multivariate conditional hazard rate (m.c.h.r.) functions  can be employed as a convenient method to describe different aspects of stochastic dependence.
 In terms of the m.c.h.r. functions, we first obtain convenient formulas for the probability distributions of the variables $\min_{j \in A} X_j$ and for the probability of events $\{X_i=\min_{j \in A} X_j\}$. Then we detail several aspects of the notion of stochastic precedence.
On these bases, we explain some controversial behavior of such variables  and give sufficient conditions under which paradoxical aspects can be excluded.
On the purpose of stimulating active interest of readers, we present several comments and pertinent examples.

\medskip
\noindent
\emph{Keywords:} Multivariate Conditional Hazard Rates, Non-transitivity, aggregation/marginalization paradoxes,   ``small" variables, initially time--homogeneous models, time--homogeneous load sharing models.

\medskip \noindent
\emph{AMS MSC 2010:} 60K10, 60E15, 91B06.
\end{abstract}

\maketitle

\section{Introduction} \label{intro}

Let us
consider a vector of non-negative random variables $\mathbf{X} =(X_{1},...,X_{n})$,
defined on a probability space $\left( \Omega ,\mathcal{F},\mathbb{P}%
\right) $. We write $[n]$ for the set $\{1,2, \ldots , n  \}$. For $i\neq j\in \left[ n\right] $, one says that $X_{i}$ \emph{is
smaller than} $X_{j}$ \emph{in the stochastic precedence} whenever the
inequality
\[
\mathbb{P}\left( X_{i} \leq X_{j}\right) \geq \mathbb{P}\left( X_{j} \leq  X_{i}\right)
\]
holds true; this condition will be denoted by $X_{i}\preceq _{sp}X_{j}$.

This notion of comparison is clearly very natural and it is of actual
interest for some applications. In fact, it had been considered several
times in the literature, possibly under a variety of different terms. In the
last few years, in particular, this property has been attracting more and
more interest in different applied contexts; see e.g. references  \cite{AS2000, BSC2004, DSS1, FHH18, NR2010}.

Several controversial or apparently counter-intuitive aspects have been
however pointed out, since a long time. In particular one can meet aspects
of \emph{non-transitivity} and other related phenomena which we will refer to as \emph{
aggregation/marginalization paradoxes}. See in particular \cite{Blyth1972,
DSS1} and the references cited therein.  More generally, it there exists a very wide literature concerning
with controversial and counter-intuitive aspects related with
non-transitivity, in mathematics and probability (see e.g. \cite{G,Savage,
steinhausparadox, TRYBULA}). In the fields of economics,
statistics, social choices, as it is well-known, the interest
toward these topics is enormous and  the
literature considering such subjects has a very long tradition, see, in particular, \cite{BF83, Fishburn,
Saari} and references cited therein. In our analysis, it is
important to be aware of the relations and similarities among all such
contexts.

The aspects
concerning with aggregation/marginalization paradoxes can be seen as related to the
literature on the theme of Simpson's paradoxes (see e.g. \cite{ Blyth2, ScaSpi,Simpson}).
Specifically concerning the topic of stochastic precedence, several examples
and counter-examples about controversial aspects can be found in the
analysis of occurrence times for ``words" in random sampling of letters from
an alphabet (see e.g. \cite{ChZa79, DSS2,G, Li80}). This field is also related to the analysis
 of stochastic comparisons for hitting times for Markov chains see e.g. \cite{DSS4, DSS5} and references therein.

Going to the specific purposes of this paper, we notice that it can be
useful to understand situations where the paradoxical phenomena of
stochastic precedence are to be expected or, on the contrary, where they can
be excluded. We  point out that many of such phenomena emerge in the
case of stochastic dependence among the random variables under consideration.
 It is relevant, in this respect, to pay attention to the way in which
stochastic dependence is described. 
Here we limit our attention to the cases,
when the joint probability distribution of $\left( X_{1},...,X_{n}\right) $
is absolutely continuous and can thus be described in terms of the joint
probability density.

More in particular we consider non-negative random variables, in which case
a possible tool  for describing the joint
probability law and the type of stochastic dependence
can be based  on the family of  \emph{multivariate conditional
hazard rate}  functions. See e.g. \cite{ShaSha90} and also
 the reviews within the more recent papers \cite{ShaSha15,Spiz15, Spiz18}. This tool is different, but
equivalent to the one based on the joint density function. In fact, there
are well-known formulas that, at least in principle, allow one to derive the
m.c.h.r. functions from the knowledge of the joint density and viceversa.
But the two types of descriptions completely differ in their abilities to
highlight different aspects of stochastic dependence. Here, we aim to point
out that, for non-negative variables, the description based upon the m.c.h.r. functions can reveal a
useful one to understand some aspects of stochastic precedence and related
issues.

The structure of the paper is described as follows.

In the next Section 2 we give some basic notation and definitions, and
preliminary results concerning the minimum among several non-negative random
variables in the jointly absolutely continuous case. In particular we
recall basic definitions and facts about the system of the m.c.h.r. functions. Section 3 will be devoted to the notion
of stochastic precedence and related controversial aspects. 
In Section 4 we analyze some
different conditions on the variables $X_{1},...,X_{n}$, that exclude the
occurrence of some of such paradoxical situations.

\section{Notation, basic definitions and preliminary results} \label{sec2}

In this section, we give basic definitions and we show some preliminary results about  the minimum among random variables. In particular we analyze the role of multivariate
conditional hazard rates.
For a given non-negative, absolutely continuous, random variable $X$, we denote by $r(t)$ the ordinary hazard rate (or failure rate) of  it:
$$
r (t ): =  \lim_{\Delta t \to 0^+}\frac{\mathbb{P} (X \in (t, t+ \Delta t )  |  X >t     )   }{\Delta t } ,
$$

To start our discussion, we recall a very simple and useful
result concerning the minimum of several independent, exponentially
distributed, random variables.

Let $\Upsilon_{1},...,\Upsilon_{n}$ denote $n$ independent random variables,
 distributed according to exponential distributions with
parameters $\lambda_{1},...,\lambda_{n}$, respectively, and set%
\begin{equation*}
\Upsilon_{1:n}:=\min\left \{\Upsilon_{1},...,\Upsilon_{n}\right \} .
\end{equation*}

Then we can state (see e.g. \cite{Norris}, Chp. 2)  the following result.

\begin{lemma}
\label{LemmaISect3}  For any $t >0 $ and $j \in [n]$,  the following identities hold
\begin{equation}
\mathbb{P}\left( \Upsilon_{1:n}=\Upsilon_{j},\Upsilon_{1:n}>t\right) =%
\mathbb{P}\left( \Upsilon_{1:n}=\Upsilon_{j})\mathbb{P}(\Upsilon
_{1:n}>t\right) ,  \label{LemmaExp(a)}
\end{equation}%
\begin{equation}
\mathbb{P}\left( \Upsilon_{1:n}=\Upsilon_{j}\right) =\frac{\lambda_{j}}{%
\sum_{s=1}^n\lambda_{s}},  \label{LemmaExp(b)}
\end{equation}%
\begin{equation}
\mathbb{P}\left( \Upsilon_{1:n}>t\right) =\exp   \left \{-t\sum_{s=1}^n\lambda_{s}    \right \}.
\label{LemmaExp(c)}
\end{equation}
\end{lemma}

We now want to show (see Proposition \ref{Ladispoli}) in which sense this result can be
extended to the  random variable   $X_{1:n}:=\min \{X_{1},\ldots ,X_{n}\}$, where             $X_1, \ldots , X_n$ are not necessarily
independent nor  exponentially distributed. We maintain
however the condition of absolute continuity for the joint probability
distribution and the joint
 density function  will be denoted by $f_{\mathbf{X}}$.
 The latter condition in particular
implies the \emph{no-tie property}
\begin{equation} \label{notie}
{\mathbb{P}}(X_{i}=X_{j})=0,
\end{equation}%
for any $i,j=1,\ldots ,n$, with $i\neq j$, which will be of
basic importance all along the paper.

We respectively denote by $f_{(1)}$, $\bar{F}_{(1)}(t)$,  $h_{(1)}(t)$,   $H_{(1)}(t)$,     the
probability density function, survival function, hazard rate function and
cumulative hazard function of   $X_{1:n}$. Namely
\begin{equation*}
h_{(1)} (t ): = \lim_{\Delta t \to 0^+}\frac{\mathbb{P} (X_{1:n} \in (t, t+ \Delta t]  |  X_{1:n} >t     )   }{\Delta t } ,
\end{equation*}
\begin{equation*}
H_{(1)}(t): = \int_0^t h_{(1)}(s) ds,
\end{equation*}
\begin{equation*}
\bar{F}_{(1)}(t) :=e^{-H_{(1)}(t)},
\end{equation*}
\begin{equation}
f_{(1)}(t) :=h_{(1)}(t)e^{-H_{(1)}(t)}.  \label{si}
\end{equation}

In view of the assumption of absolute
continuity, we can define the following limits, for $j=1, \ldots , n $
\begin{equation} \label{alfaconi}
\gamma_j (t) := \lim_{\Delta t \to 0^+}{\mathbb{P}} (X_j = X_{1:n} | X_{  1:n  }
\in (t ,t+ \Delta t ] ) = {\mathbb{P}}(X_j = X_{1:n} | X_{1:n} =t ),
\end{equation}
\begin{equation}
\mu _{j}(t):= \lim_{\Delta t\rightarrow 0^+}\frac{1}{\Delta t}%
\mathbb{P}(X_{j}\leq t+\Delta t|X_{1:n}>t).
\end{equation}
We notice that, in the case of  regular conditional probabilities,
\begin{equation}  \label{somma}
\sum_{j= 1}^n \gamma_j (t ) =1 .
\end{equation}
Furthermore,
\begin{equation}
\mu _{j}(t)=h_{(1)}(t)\gamma _{j}(t)  \label{miconI} .
\end{equation}
In fact,  since
$$
\mathbb{P }(     \{ X_j  \leq t+\Delta t \}  \cap  \{ X_{1:n } > t\}  ) = \mathbb{P }(     \{ X_j  \leq t+\Delta t \}  \cap  \{ X_{1:n } \in (t, t +\Delta t]\} )
$$
we can write
\begin{align*}\label{perY}
h_{(1)}(t)\gamma _{j}(t) &=
\lim_{\Delta t \to 0^+} \frac{\mathbb{P} (\{ X_j = X_{1:n}\}   \cap \{ X_{1:n} \in (t,t +\Delta t ] \}  )} {\mathbb{P} ( X_{1:n} \in (t,t +\Delta t ]   ) }
\frac{\mathbb{P} ( X_{1:n} \in (t,t +\Delta t ]   )} {      \Delta t \mathbb{P} ( X_{1:n} > t   ) }
\\
&=\lim_{\Delta t \to 0^+} \frac{\mathbb{P} (\{ X_j < t + \Delta t \}   \cap \{ X_{1:n}> t \}  )} { \Delta t \mathbb{P} ( X_{1:n} >t ) }
= \mu_j (t).
\end{align*}

\medskip \medskip

The following two results will have a key role in the next discussion.
\begin{prop}
\label{Ladispoli}
With the notation introduced above, the following properties hold.
\begin{itemize}
\item[a)] For any $t \geq 0 $ one has
\begin{equation}  \label{parteA}
H_{(1)}(t) =      \sum_{i=1}^n     \int_0^t  \mu_i (s) ds .
\end{equation}

\item[b)] For any Borel set $B \in \mathcal{B}(\mathbb{R_+})$, one can write
\begin{equation}  \label{second}
{\mathbb{P}}( X_{i} = X_{1:n} , X_{1:n} \in B ) = \int_{B} \mu_i (s)
e^{-H_{(1)}(s)} ds .
\end{equation}
\end{itemize}
\end{prop}

\begin{proof}
From \eqref{somma} and \eqref{miconI} one has $h_{(1)} (s)=\sum_{i=1}^n \mu_i
(s) $  and then \eqref{parteA}.
Taking into account the positions \eqref{si} and \eqref{miconI} we obtain
$$
{\mathbb{P}}( X_{j} = X_{1:n} , X_{1:n} \in B ) = \int_{B} f_{(1)} (s ) \mathbb{P} ( X_j = X_{1:n} | X_{1:n} =s )ds =
$$
\begin{equation}  \label{secondproof}
 \int_{B} h_{(1)} (s ) e^{-H_{(1)}  (s)} \gamma_i (s)ds,
\end{equation}
that is equal to \eqref{second}.
\end{proof}

 Denote by $\Lambda $  the Lebesgue measure on $ (\mathbb{R}_{+},\mathcal{B}(\mathbb{R_{+}}))$.

\begin{theorem}
\label{th:1} The following two statements are equivalent:

\begin{itemize}
\item[a)] For each $B \in \mathcal{B}(\mathbb{R_+})$ and $i, j = 1, \ldots n
$,
$$
{\mathbb{P}}( X_{1:n} \in B , X_{i} = X_{1:n} ) \leq {\mathbb{P}}(
X_{1:n} \in B , X_{j} = X_{1:n} );
$$
\item[b)] $\Lambda (\{ t\in \mathbb{R}_+ : \mu_i (t) > \mu_j (t)\}) =0$.
\end{itemize}
\end{theorem}

\begin{proof}
The implication b) $\Rightarrow $ a) is immediate in view of the identity %
\eqref{second}.

The implication a) $\Rightarrow $ b) is proved by contradiction.

Assume, in fact, $\Lambda (\{ t\in \mathbb{R}_+ : \mu_i (t) > \mu_j (t)\})
>0 $ then, by continuity of probability measures, there exists $\varepsilon >0 $ such
that $\Lambda (\{ t\in \mathbb{R}_+ : \mu_i (t) > \mu_j (t) + \varepsilon
\}) >0 $. Therefore, by setting
$$
B = \{ t\in \mathbb{R}_+ : \mu_i (t) >
\mu_j (t) + \varepsilon \}
$$
one obtains
\begin{equation*}
{\mathbb{P}}( X_{i} = X_{1:n} , X_{1:n} \in B ) - {\mathbb{P}}( X_{j} =
X_{1:n} , X_{1:n} \in B ) \geq \varepsilon \int_{B} e^{-H_{(1)}(s)} ds >0 .
\end{equation*}
\end{proof}

It is convenient, at this step, to recall the definition of m.c.h.r. functions for the non-negative random variables $X_{1},...X_{n}$. 
We denote by $
X_{1:n},...,X_{n:n}$ the corresponding order statistics.
For $A\subseteq\left[  n\right]  $ with $|A|>1$, set
\[
X_{1:A}:=\min_{i\in A}X_{i} .
\]

In particular,  we obtain
\begin{equation*}
X_{1:[n]}:=X_{1:n}=\min_{1\leq j\leq n}X_{j}.
\end{equation*}
In the following definition for a given subset $I \subset [n]$ we will consider the random variable $  X_{1:\tilde I}  $, where the symbol $\tilde I $  denotes the complementary set $[n] \setminus I$.

\begin{definition} \label{defi1}
For a fixed index $j\in\lbrack n]$, an ordered set $I=(i_{1},...,i_{k})$ $%
\subset\lbrack n]$ with $j\notin I$, and an ordered sequence $0<
t_{1}<...< t_{k}$, the Multivariate Conditional Hazard Rate function $%
\lambda_{j}(t|I;t_{1},...,t_{k})$ is defined as follows:
\begin{equation}
\lambda_{j}(t|I;t_{1},...,t_{k}):=\lim_{\Delta t\to 0^+}\frac{1}{\Delta
t}\mathbb{P}(X_{j}\leq t+\Delta
t|X_{i_{1}}=t_{1},...,X_{i_{k}}=t_{k},    X_{1: \tilde I}    >t).  \label{DefMCHR}
\end{equation}
Furthermore, one puts%
\begin{equation}
\lambda_{j}(t|\emptyset):=\lim_{\Delta t\to 0^+}\frac{1}{\Delta t}
\mathbb{P}(X_{j}\leq t+\Delta t|     X_{1:n}   >t).  \label{DefMCHR2}
\end{equation}
\end{definition}
For what specifically concerns the position in \eqref{DefMCHR2}, we must notice that we reobtain nothing else than the functions defined in \eqref{miconI}; more precisely
\begin{equation}\label{uguali}
\mu_j (t) = \lambda_j (t | \emptyset ) .
\end{equation}
For this reason, the symbol $\mu_j (t)$ will not be used anymore, from now on.

\begin{rem} \label{general}
The limits considered in the above definition make sense in view of the
assumption of absolute continuity and the quantity $%
\lambda_{j}(t|I;t_{1},...,t_{k})$ can be seen as the failure intensity, at
time $t$, associated to the conditional distribution of the variable $X_{j}$, given the observation of the \textit{dynamic history}
\begin{equation}
\mathfrak{h}_{t}=:\{X_{i_{1}}=t_{1},...,X_{i_{k}}=t_{k},     X_{1:\tilde I}      >t\}.
\label{DynHistory}
\end{equation}

The functions $\lambda_{j}(t|I;t_{1},...,t_{k})$ and $\lambda_{j}(t|%
\emptyset)$ can be computed in terms of the joint density function $%
f_{\mathbf{X}}$. On the other hand, based on the knowledge of the
functions $\lambda_{j}(t|I;t_{1},...,t_{k})$ and $\lambda_{j}(t|\emptyset)$,
one can recover the  function $  f_{\mathbf{X}}       $.
In fact, the following formula holds for $0< x_{1}<...< x_{n}$
\begin{equation*}
  f_{\mathbf{X}}        \left( x_{1},...,x_{n}\right)
=\lambda_{1}(x_{1}|\emptyset)\exp \left \{-    \sum_{j=1}^{n}    \int_{0}^{x_{1}}
   \lambda_{j}(u|\emptyset) du   \right \}\times
\end{equation*}%
\begin{equation*}
\times \lambda_{2}(x_{2}|\{1\};x_{1})\exp   \left   \{-   \sum_{j=2}^{n}    \int_{x_{1}}^{x_{2}}
\lambda_{j}(u|\{1\};x_{1}) du   \right \}\times...
\end{equation*}
\begin{equation*}
\times\lambda_{k+1}(x_{k+1}|\{1,...,k\};x_{1},...,x_{k})\exp \left \{-    \sum_{j=k+1}^{n}
\int_{x_{k}}^{x_{k+1}}
   \lambda_{j}(u|\{1,...,k\};x_{1},...,x_{k})
du   \right \}\times...
\end{equation*}
\begin{equation}
\times\lambda_{n}\left( x_n|\{1,...,n-1\};x_{1},...,x_{n-1}\right) \exp
\left \{-\int_{x_{n-1}}^{x_{n}}\lambda_{n}(u|\{1,...,n-1\};x_{1},...,x_{n-1})du     \right \}.
\label{JointDensInTermsmchr}
\end{equation}

Similar expressions hold when $x_{1},...,x_{n}$ are such that $x_{\pi
(1)}<...< x_{\pi (n)}$, for some permutation $\pi $. For proofs,
details, and for general aspects see \cite{ShaSha90}, \cite{ShaSha}, and the
review paper \cite{ShaSha15}.
\end{rem}

\begin{rem} \label{R2}
In the reliability field, the variables $X_{1},...,X_{n}$ are interpreted as
the random lifetimes of $n$ components in a system. From the identities \eqref{miconI}
and \eqref{uguali} one immediately obtains the relation
\[
\lambda_{j}(t|\emptyset)=h_{(1)}(t)\gamma_{j}(t),
\]
tying the specific m.c.h.r. function $  \lambda_{j}(t|\emptyset)  $ with the
conditional probability
$    \gamma_{j}(t) $ defined in \eqref{alfaconi} and with the
univariate failure rate of the minimum $X_{1:n}$, i.e. the lifetime of the series system.

In the frame of system reliability, indexes of importance of a component are
generally relevant notions. In that context, the conditional probabilities  $\lambda_{j}(t|\emptyset)$ and $ \gamma_{j}(t) $ are related with the Barlow-Proschan indexes of importance of
stochastically dependent components (see \cite{BarPro1975,  Iyer1992, MaMa2013, MN2019}) in the special case when the system is a
series. An integral expression for Barlow-Proschan index in a series system can be obtained by specializing a result
given in  \cite{MN2019}. Such an expression is  alternative to the integral one given in \eqref{secondproof} above,
with $B= \mathbb{R}_+$.   In  \cite{MN2019},
 the validity of  expression \eqref{somma} have been pointed out for
general coherent systems.

We notice that some arguments in \cite{MN2019} might also be applied to conditional probabilities of the type $\mathbb{P}(X_{j}
=X_{k+1:n}|X_{k+1:n}=t;\mathfrak{h}_{t})$ which are related with conditional importance
indexes for surviving components in a series system surviving at time $t$.

\end{rem}

\begin{rem}    \label{R3}
Let us  look at the conditional distribution of the residual lifetimes of
components surviving at time $t$ given the dynamic history $\mathfrak{h}_{t}$
in \eqref{DynHistory}.  Similarly to what noticed for $  \lambda_{j}(t|\emptyset)   $ in the remark above, and conditioning upon the observation
$\mathfrak{h}_{t}$, one can obtain an expression for the m.c.h.r. function
$ \lambda_{j} (t|I;t_{1},...,t_{k})$ in terms of the conditional probability
$$
\mathbb{P}(X_{j}=X_{k+1:n}|X_{k+1:n}=t;\mathfrak{h}_{t})
$$ and of the
conditional univariate failure rate of $X_{k+1:n}$ (namely, the residual
lifetime of the series system made with the components  surviving at time
$t$).

On the other hand, the m.c.h.r. functions $\lambda_{j} (t|I;t_{1},...,t_{k}) $ can be used to
describe  the conditional distribution of the residual lifetimes of
components,  given  $\mathfrak{h}_{t}$. Such a description  will  be presented
in formula \eqref{lamba}. It can be interesting to compare it
with the alternative expression     that can be given in terms of copula-based
representations of joint distributions of lifetimes (see in particular
\cite{DF2014} and \cite{DN2017}).
\end{rem}

\bigskip

 As a direct corollary of Proposition \ref{Ladispoli} we obtain that, for any vector of dependent variables, probabilities of events related to the behavior of their minimum are equal to probabilities of corresponding events for a vector of independent variables.
 We point out that, in the case of independence, the function $\lambda_{j}(\cdot|\emptyset)$ coincides with the ordinary failure rate functions $r_{j} (\cdot )$ and
 we can more precisely state the following results.
\begin{prop} \label{prop-ind}
Let $(X_1, ...,X_n)$ be a vector with m.c.h.r. functions $\lambda_{j}(t|\emptyset)$
and take independent random variables $Z_{1},...,Z_{n}$, with ordinary failure rate functions $r_{j}$ given by
\[
r_{j}(t):=\lim_{\Delta t\to  0^+} \frac{
\mathbb{P}\left(  Z_{j}<t+\Delta t|Z_{j}>t\right)}{\Delta t} = \lambda_{j}(t |\emptyset) .
\]
Then
\begin{equation} \label{indipendente}
\mathbb{P} ( X_{i} = X_{1:n} ,  \,\,  X_{1:n} \in B   ) = \mathbb{P} ( Z_{i} = Z_{1:n} , \,\,   Z_{1:n} \in B   )
\end{equation}
for any   $i \in [n ]$ and any Borel set $ B$. 
\end{prop}
\begin{proof} In order to prove \eqref{indipendente} it is enough to apply, to both the vectors of random variables $(X_1, \ldots , X_n)$
and $(Z_1, \ldots , Z_n)$, item b) of Proposition \ref{Ladispoli}.
\end{proof}
For our purposes it is also useful to specialize Proposition \ref{Ladispoli} and Proposition \ref{prop-ind} to the limiting case $B = [0, \infty )$. Thus we obtain that the analysis of the minimum among several random variables can be reduced to the case of independent variables by means of the functions $\lambda_{1}(s|\emptyset) , \ldots , \lambda_{n}(s|\emptyset) $ by
considering the following formulas.
\begin{equation}
\mathbb{P}\left(  X_{1:n }>t\right)  =\exp \left \{-     \sum_{i=1}^{n}         \int_{0}^{t}
     \lambda_{i}(s|\emptyset)ds  \right \}  =\exp \left \{-   \sum_{i=1}^{n}           \int_{0}^{t}
     r_{i}(s)ds  \right \} = \mathbb{P}\left(  Z_{1:n }>t\right)
,\label{Exp(c)Gen}
\end{equation}
\begin{equation}\label{Rome}
\begin{aligned}
{\mathbb{P}}( X_{i} = X_{1:n}  ) &= \int_{0}^{\infty}  \lambda_{i}(s|\emptyset)
\exp\{-    \sum_{i=1}^{n}\lambda_{i}(s|\emptyset)    \} ds
\\
&= \int_{0}^{\infty} r_{i}(s)
\exp\{-    \sum_{i=1}^{n}r_{i}(s)    \} ds = {\mathbb{P}}( Z_{i} = Z_{1:n}  )
\end{aligned}
\end{equation}
for any $i \in [n]$.

\medskip

Before continuing, the method of m.c.h.r. functions for
describing the behaviour of the minimum among dependent variables  will now
 be further demonstrated by means of some relevant examples.

\begin{example}
\label{remo3}  (The case of exchangeability)
When $X_{1},...,X_{n}$ are exchangeable, then the dependence of $\lambda
_{j}(t|\emptyset)$ on the index $j$ is obviously dropped, namely for a
suitable function $\lambda(\cdot|\emptyset)$ and for $j=1,...,n$, $t>0$ $,$%
\begin{equation}
\lambda_{j}(t|\emptyset)=\lambda(t|\emptyset). \label{ExchCond}%
\end{equation}

Thus we obtain%

\begin{equation}
\mathbb{P}\left(  X_{1:n }>t\right)  =\exp    \left \{-n\int_{0}%
^{t}\lambda(s|\emptyset)ds \label{Exp(c)Exch} \right \}   ,
\end{equation}

\bigskip%

\begin{equation}
\mathbb{P}\left(  X_{1:n  }=X_{j}\right)  =\frac{1}{n}
\label{Exp(b)Exch}.
\end{equation}

Notice that the same identities do hold even if $X_{1},...,X_{n}$ are not
exchangeable, provided the above condition \eqref{ExchCond} holds.
\end{example}

\begin{example} \label{remerei}
 (The case of conditional independence and identical exponential distribution).
Let $\Theta$ be a non-negative random variable with distribution $\Pi_{\Theta
}$ and let $X_{1},...,X_{n}$ be conditionally independent and exponentially
distributed given $\Theta$, i.e.
\[
\mathbb{P}\left(  X_{1}>x_{1},...,X_{n}>x_{n}\right)  =\int_{0}^{\infty}%
\exp\{-\theta\sum_{i=1}^{n}x_{i}\}\Pi_{\Theta}\left(  d\theta\right)  .
\]

In this case one has (for details see e.g. \cite{SpiBook})
\[
\lambda_{j}(t|\emptyset)=                     \mathbb{E}(\Theta|X_{ 1:n  }>t) = \int_{0}^{\infty}\theta\Pi_{\Theta}\left(
d\theta|X_{1:n }>t\right)  ,
\]
where $\Pi_{\Theta}\left(  \cdot|X_{1:n }>t\right)  $ denotes
the a posteriori distribution of $\Theta$, given the observation $
X_{1:n  }>t $. Moreover%

\[
\mathbb{P}\left(  X_{1:n  }>t\right)  =\int_{0}^{\infty}%
\exp\{-n \theta t  \}\Pi_{\Theta}\left(  d\theta\right)
\]
and since $X_{1},...,X_{n}$ are, in particular, exchangeable%

\[
\mathbb{P}\left(  X_{1:n }=X_{j}\right)  =\frac{1}{n}.
\]
\end{example}

\begin{example}
\label{remo4} The following case can be considered as a generalization of the case of
independent, exponential, variables: consider dependent random variables  $\Upsilon_{1},...,\Upsilon_{n}$
such that for $j=1,...,n$, the ratio
\[
\frac{\lambda_{j}(t|\emptyset)}{\sum_{i=1}^{n}\lambda
_{i}(t|\emptyset)}%
\]
does not depend on the variable $t$. In such a case, the identities
(\ref{LemmaExp(a)}) and (\ref{LemmaExp(b)}) hold.
\end{example}

\begin{example}
\label{esempioFABIO}
A more special class of survival models generalizing the case of independent,
exponential, variables is the one of \emph{time-homogeneous load-sharing}
models, characterized by the condition
\[
\lambda_{j}(t|\emptyset)=r_{j}\left(  \emptyset\right)  ,\lambda_{j}%
(t|I;t_{1},...,t_{k})=r_{j}\left(  I\right)  ,
\]
for a suitable family of constants $\{r_{j}\left(  \emptyset\right)
;r_{j}\left(  I\right)  ;j\in\left[  n\right]  ,I\subset\left[  n\right]
,i\notin I\}$.

Several theoretical and applied aspects of such survival models have been
studied in different fields and, in particular, in the reliability literature.
See e.g. \cite{Spiz18} and references cited therein.

By limiting attention to this class of models, useful examples can be
constructed for different types of properties related with the arguments of
this paper.  In particular, in the case of time-homogeneous load-sharing model
we obtain from Proposition 1
\[
\mathbb{P}\left(  X_{1:n}=X_{j},X_{1:n}>t\right)  =r_{j}(\emptyset
)\exp\{-t\sum_{i=1}^{n}r_{i}\left(  \emptyset\right)  \}.
\]
\end{example}

Still considering time-homogeneous load-sharing models,
it is also useful recalling attention on the following property of conditional
distribution of the residual lifetimes
\[
\ X_{j_{1}}-t,...,X_{j_{n-k}}-t
\]
given the observation of a dynamic history $\mathfrak{h}_{t}$ as in
\eqref{DynHistory}. Of course, conditionally on $\mathfrak{h}_{t}$, the
joint distribution of $\left(  X_{j_{1}}-t,...,X_{j_{n-k}}-t\right)  $ is
generally absolutely continuous if the one of $\left(  X_{1},...,X_{n}\right)
$ is such. Furthermore it is a time-homogeneous load-sharing model if joint
distribution of $\left(  X_{1},...,X_{n}\right)  $ is such and one has the
simple relation%
\[
\widehat{r}_{j}(\emptyset)=r_{j}(I),j\in\widetilde{I}.
\]

From Proposition 1 we obtain, for $j\in\widetilde{I}$,%

\[
\mathbb{P}\left(  X_{\tilde{I} }=X_{l},X_{1:
\tilde{I} }>t+s|\mathfrak{h}_{t}\right)  =r_{l}(I)\exp \left
\{-s\sum_{j\in\tilde{I}}r_{j}\left(  I\right)  \right  \}.
\]

Denote by
 $J_{1},J_{2},...J_{k}$ the random  indices such that
\[
X_{1:n}=X_{J_{1}},...,X_{k:n}=J_{k}.
\]
By applying the product formula of conditional probabilities, we thus can also
obtain that the joint density function
\[
f_{X_{1:n},\ldots,X_{k:n},J_{1},\ldots,J_{k}}(t_{1},t_{2},\ldots,t_{k}%
,j_{1},j_{2},\ldots,j_{k})
\]
of $(X_{1:n},\ldots,X_{k:n},J_{1},\ldots,J_{k})$, $k=1,\ldots,n$, with respect
to the product of $k$-dimensional Lebesgue measure on $[0,\infty)^{k}$ and
$k$-dimensional counting measure on $[n]  ^{k}$ is the
product of terms of the form
\begin{equation} \label{factori}
r_{j_{h+1}}(\{j_{1},\ldots,j_{h}\})\exp\{-(t_{h+1}-t_{h})\sum_{l\neq
j_{1},\ldots,j_{h}}r_{l}\left( \{ j_{1},\ldots,j_{h}   \}   \right)  \}.
\end{equation}

Using once again Proposition 1, the argument presented  above
can easily be extended to the case of an arbitrary absolutely continuous
model, characterized in terms of its m.c.h.r. functions.

First of all we notice that, conditionally on a dynamic history $\mathfrak{h}_{t}$, the
joint distribution of residual lifetimes $\left(  X_{j_{1}}-t,...,X_{j_{n-k}}-t\right)  $ is
characterized by the m.c.h.r. functions
\begin{equation}\label{lamba}
\widehat{\lambda}_j^{( \mathfrak{h}_{t}  )} (t | \emptyset ) = \lambda_j (t | I; t_1, \ldots , t_h)  , \,\,\, j \in \tilde I .
\end{equation}
We can thus state the following proposition.
\begin{prop}\label{jointdensityordstataug}
The joint density function
\[
f_{X_{1:n},\ldots,X_{k:n},J_{1},\ldots,J_{k}}(t_{1},t_{2},\ldots,t_{k}%
,j_{1},j_{2},\ldots,j_{k})
\]
of $(X_{1:n},\ldots,X_{k:n},J_{1},\ldots,J_{k})$, $k=1,\ldots,n$, with respect
to the product of $k$-dimensional Lebesgue measure on $[0,\infty)^{k}$ and
 counting measure on $[n]  ^{k}$ is the
product of terms of the form
\begin{equation}\label{factors}
\lambda_{j_{h+1}}(t_{h+1};\{j_{1},\ldots,j_{h}\}; t_1, \ldots, t_h)\exp\{-(t_{h+1}-t_{h})\sum_{l\neq
j_{1},\ldots,j_{h}}\lambda_{l}\left(   t_{h+1};\{j_{1},\ldots,j_{h}\}; t_1, \ldots, t_h       \right) \} .
\end{equation}
\end{prop}

As Proposition \ref{prop-ind} shows, the factors in \eqref{factors} can be replaced, at any step, by corresponding factors related with independent variables whose distribution are affected by the past observations.

For time-homogeneous load--sharing models, the factors in \eqref{factors} reduce to those in \eqref{factori}. The concept of time-homogeneous load--sharing models can be  extended in a natural way to the non--homogeneous case. For such a case, the specific form of the above result has been  given in \cite{RycSpiz}.


Before concluding this section, we also recall attention on a further aspect of m.c.h.r. functions.  For $m <n $,  m.c.h.r. functions are generally different from the corresponding  m.c.h.r. functions
associated to the marginal distribution of the vector $(X_1 , \ldots , X_m)$.

\section{ Controversial aspects of stochastic precedence  } \label{sec3}

Let $Y_1$ and $Y_2$ be two random variables.
We remind from the Introduction  that
 $Y_1$ stochastically
precedes $Y_2$ if $\mathbb{P}(Y_1 \leq Y_2) \geq \mathbb{P}(Y_2 \leq Y_1)$. Under the  no-tie condition,
this definition is equivalent to   $\mathbb{P}(Y_1 \leq Y_2) \geq \frac{1}{2}$.  This will be written $Y_1
\preceq_{sp} Y_2$.
 The previous formula \eqref{Rome}  in particular provides us with a simple
characterization of stochastic precedence when $Y_1 $ and $Y_2$ are two non-negative random variables. In fact, letting
\begin{equation*}
\lambda_{i}(t|\emptyset):=\lim_{\Delta t\to 0^+}\frac{1}{\Delta t}
\mathbb{P}(Y_{i}\leq t+\Delta t|     Y_{1:n}   >t)
\end{equation*}
for $i = 1,2$,
 we can write
\begin{equation}  \label{perduevariabili}
 Y_1 \preceq_{sp} Y_2      \Leftrightarrow
  \int_{0}^{+\infty}        \lambda_1 (s | \emptyset )
e^{-  H_{(1)} (s)      } ds  \geq \frac{1}{2}
\end{equation}
where $H_{(1)} (t)    =  \int_0^{t}     [    \lambda_1 (s | \emptyset ) +  \lambda_2 (s | \emptyset )      ]      \,   ds  $.

\begin{example}   \label{romolo}    (The case of independence).
Let $X_{1},X_{2}$ be two independent, non-negative, random variables with
absolutely continuous distributions characterized by the hazard rate functions
$r_{1}(t),r_{2}(t)$, respectively. Then, in view of the
characterization in \eqref{perduevariabili},  one has

\[
X_{1}   \preceq_{sp}
X_{2}\Leftrightarrow      \int_{0}^{\infty}r
_{1}(t)e^{-\int_{0}^{t}\left[  r_{1}(s)+r_{2}(s)\right]  ds}dt \geq  \frac{1}{2} .
\]

When $X_{1},X_{2}$ are independent and exponential with parameters
$r_{1},r_{2}$, the condition $X_{1}\preceq_{sp}X_{2}$ simply
becomes $r_{1}\geq r_{2}$  and so we get the hazard rate ordering.
\end{example}

\bigskip

\begin{example} \label{secondoE}      (The case of conditional independence and exponentiality).
Similarly to the previous Example \ref{remerei}, consider now the case when $\Theta$ is a non-negative random variable with distribution $\Pi_{\Theta
}$
and $X_{1},  X_{2}$ are conditionally independent given $\Theta$, with
$$
\mathbb{P}(X_{1}>t|\Theta=\theta)=\exp\{-c_{1}\theta    t     \},
$$
$$
\mathbb{P}%
(X_{2}>t|\Theta=\theta)=\exp\{-c_{2}\theta     t \},
$$
where $c_{1},c_{2}$ are two fixed positive numbers. In this case, one has
\[
\lambda_{i}(t|\emptyset)=c_{i}\mathbb{E}(\Theta|X_{ 1:n  }>t).
\]
Thus, we are in the case of Example \ref{remo4} and  the condition $X_{1}\preceq_{sp}X_{2}$
becomes $c_1 \geq c_2$.
\end{example}

\begin{rem} \label{rema}
It is immediate to see that, in the case of stochastic independence, the
condition $X_{1}\preceq_{sp}X_{2}$ is implied by the condition that
$X_{1}$ precedes $X_{2}$ in the usual stochastic ordering (written
$X_{1}\preceq_{st}X_{2}$), namely%
\[
\mathbb{P}(X_{1}>t)\leq\mathbb{P}(X_{2}>t),\forall t>0.
\]

This implication is not valid anymore, when the condition of independence is
dropped; see e.g. the discussion and counter-examples in \cite{DSS1, DSS4}. The characterization in \eqref{perduevariabili} can help us to easily
understand the logic on which  counter-examples may be built up. 
\end{rem}

\begin{rem} \label{remava}
The relation of stochastic precedence does not generally satisfy the
transitivity property. In fact, it is possible to show examples where, for
three  real-valued random variables $X_{1},X_{2},X_{3}$, the following
conditions simultaneously hold:%
\begin{equation} \label{transitivity}
\mathbb{P}(X_{1}<X_{2}) >\frac{1}{2}, \,\,\mathbb{P}(X_{2}<X_{3}) >\frac{1}{2}, \,\,\mathbb{P}(X_{3}<X_{1}) >\frac{1}{2}
\end{equation}

Possibly under different languages, this topic has been often considered in the literature and famous examples have been given (see e.g. \cite{Blyth1972, Fishburn, G, Savage, steinhausparadox, TRYBULA}). In this respect, we point out that, for the case of non-negative variables, examples in discrete-time can be easily converted into examples in continuous-time.

 We notice furthermore that the possibility of \eqref{transitivity} is obviously excluded when $X_{1}%
,X_{2},X_{3}$ are independent variables, satisfying the property%
\[
X_{1}\preceq_{st}X_{2}\preceq_{st}X_{3}.
\]
\end{rem}

We now introduce the following notation to point out a further aspect, of  stochastic precedence, which may
appear controversial at first glance.

The probability  $\mathbb{P}\left(  X_{1:n}=X_{j}\right) $ will be denoted by $\alpha_{j}$.
More generally, for $A\subseteq\left[  n\right]  $ with $|A|>1$, $j\in A$, we set
\begin{equation} \label{nuovoalfa}
\alpha_{j}^{[A]}:=\mathbb{P}\left(  X_{1:A}=X_{j}\right) .
\end{equation}
For $A\subseteq\left[  n\right]  $ with $|A|>1$, $i,j\in A$, set
\begin{equation} \label{spinA}
X_{i}\preceq_{sp}^{\left[  A\right]  }X_{j}%
\end{equation}
if
\[
\alpha_{i}^{[A]}\geq \alpha_{j}^{[A]}.
\]

 In words the relation \eqref{spinA}  says that the random variable $X_i $ has, with respect to $X_j$, a greater (or equal) probability
to be the minimum among all the variables $X_k$ with $k \in A \subset [n]$. Such a relation can be of interest in several contexts  where it
is important to detect which variable has a greater probability to be a minimum among a given set of random variables.
For instance, consider the collection, denoted by $[n]$, of all the horses registred in a racecourse. Let $A\subset [n]$ be
the set of those  horses scheduled to take part in a specific race. The bookmakers are called to compare the probability of victory only among the horses belonging to $A$. Thus they are interested in the probabilities  $\alpha_i^{[A]} $'s, for $i \in A$,  and in the relations \eqref{spinA}.
It is natural to think of  others  contexts, in particular Economics and Reliability theory, in which
definition \eqref{spinA} could be relevant, see e.g. Example \ref{esempioSpi}  below.

\medskip

Consider now two non-disjoint subsets $A,B\subseteq [n]$ and two elements
$i,j\in A\cap B$.

We notice that the inequalities
\[
X_{i}\preceq_{sp}^{\left[  A\right]  }X_{j},X_{j}\preceq_{sp}^{\left[  B\right]
}X_{i}%
\]
can simultaneously hold (see also, e.g., \cite{Saari}). In particular, it can happen that, for an element
$l\notin A$, we can have%
\[
\alpha_{i}^{[A]}>\alpha_{j}^{[A]},\alpha_{i}^{[A\cup l]}<\alpha_{j}^{[A\cup
l]]}.
\]
and, for three different elements $i,j,l\in\left[  n\right]  $,
\[
X_{i}\preceq_{sp}X_{j},X_{j}\preceq_{sp}^{\left[ \{ i,j,l \}\right]  }X_{i}.
\]
Example \ref{no2} in the next section shows a case where the latter   situation arises.
We will refer to this type of circumstances as to an \textit{aggregation/marginalization
paradox}.

\begin{example} \label{esempioSpi}
Let $T$ denote the lifetime of a \textit{coherent} system made with $n$ binary
components, whose lifetimes are denoted by $X_{1},...,X_{n}$. As it is very
well known, $T$ can  be written in the form%

\begin{equation} \label{cutsets}
T=\max_{k=1, \ldots , K}\min_{h\in P_{k}}X_{h}
\end{equation}
where $P_{1},...,P_{K}$ are the\textit{ minimal path sets }of the system (see
e.g. \cite{BarPro1975}). Any $P_{k}$ is a series system and Equation \eqref{cutsets} says
that the system can be written as a parallel of series system. Using the
notation defined in \eqref{nuovoalfa}, the
quantity $\alpha_{h}^{[P_{k}]}$ can be seen as the Barlow-Proschan importance
index of the component $h$ w.r.t. the $k$-th minimal path set. See also Remark
2 in Section 2. For two components $i$ and $j$ both belonging to two different
minimal path sets $P_{k^{\prime}}$ and $P_{k^{\prime\prime}}$, it may in
general happen that $\alpha_{i}^{[P_{k^{\prime}}]}<\alpha_{j}^{[P_{k^{\prime}%
}]}$ and $\alpha_{i}^{[P_{k^{\prime\prime}}]}>\alpha_{j}^{[P_{k^{\prime\prime
}}]}$. Such a situation may appear rather controversial, taking into account
that both $P_{k^{\prime}}$ and $P_{k^{\prime\prime}}$ are series systems.
\end{example}


\section{A simplifying scenario} \label{sec4}

Let the random vector $\mathbf{X} =(X_1, \ldots , X_n)$  be given.
An issue of interest in our study is the identification
of the variables that are small according to the following definition

\begin{definition}
\label{small} 
\begin{itemize}
\item[(i)] We say that $X_i$ is \emph{weekly small} with respect to (w.r.t.) $\mathbf{X} =  (X_1,X_2,\ldots,X_n)  $ if $\alpha_i
\geq \alpha_j$ for $j = 1 , \ldots , n $.

\item[(ii)] We say that $X_i$ is \emph{small} w.r.t. $\mathbf{X}$ if $X_i$ is
weekly small w.r.t. $\mathbf{X}$ and there exists $j$ such that $%
\displaystyle {\alpha_i> \alpha_j}$.
\end{itemize}
\end{definition}
Notice that a weakly small element always exists whereas the existence of a small element is not guaranteed. However,
 no small element exists if and only if $\alpha_1 = \cdots = \alpha_n = 1/n$.

 Actually, the quantities $(\alpha_i: i =1, \ldots ,n)$  can be  computed using \eqref{Rome}. However, for some models, the determination of small variables may be rather complicate.

In this section we analyze situations where the scenario is simplified.
First of all, we  notice that, in the case when we deal with only two random variables, the property of being weakly small is actually equivalent to stochastic precedence.

Let us now consider the case when $n>2 $. As the following example shows, the circumstance that $X_1$ stochastically precedes all the variables $
X_2,....,X_n$ does not imply (and is not implied by) the condition that $X_1$ is small  w.r.t. $\mathbf{X}$.
\begin{example}\label{no2}
Let  $X_1,X_2,X_3$ be three independent random variables where, for
    $\varepsilon \in  (0,\frac{1}{2}) $,
$X_1 $ is the degenerate random variable $\frac{1}{2} - \varepsilon $ and where $X_2, X_3 \sim  U(0,1)$.
 For $\varepsilon$  small enough, the r.v. $X_2$ and $X_3$ are small w.r.t. $(X_1,X_2,X_3)$.
Indeed
$$
    \mathbb{P}    (X_1 < \min \{X_2,X_3\}) =\left  (\frac{1}{2} + \varepsilon \right )^2 \cong \frac{1}{4}
$$
and
$$\mathbb{P} (X_2 < \min \{X_1,X_3\}) = \mP(X_3 < \min \{X_1,X_2\}) =\frac{1}{2}\left [1-   \mathbb{P}    (X_1 < \min \{X_2,X_3\}) \right ]=
$$
$$
=\frac{1}{2}\left [1- \left (\frac{1}{2}+\varepsilon \right )^2 \right ]\cong \frac{3}{8}  .
$$

On the other hand, we obviously have
  $$
  \mathbb{P}(X_1 < X_2) = \mathbb{P}(X_1 < X_3) =  \frac{1}{2}+ \varepsilon  .
  $$

\end{example}

Of course, checking the stochastic precedence of a random variable $X_1$  with respect to a set of other variables is typically much easier than checking the property of it being small.
In this respect, the following two  definitions are of interest in our analysis.

Given  $\mathbf{X}=\left (X_1,X_2,\ldots,X_n
\right )$, for any $i = 1, \ldots , n $ we denote by $      V_{[i]}  $ the set of indexes defined as
\begin{equation*}
V_{[i]} = \left \{j \in [n]: {\mathbb{P}} (X_i < X_j) \geq \frac{1}{2}
\right \},
\end{equation*}
and set
$$
     \mathbf{X}_{A}    =(X_i: i \in A).
$$
\begin{definition}
\label{couple} We say that $X_i$ is \emph{pair-determined} in $\mathbf{X}$ if for
any subset $A\subset V_{[i]} $ the random variable $X_i $ is weakly small
w.r.t. $\mathbf{X}_{A \cup \{i\}}           $. 
\end{definition}

The applied meaning of the above definition can be
appreciated by thinking of a betting situation, where $X_{1},...,X_{n}$ are
hitting times until the first occurrence of competing events (such as in
horse-racing) and where different players are expected to bet on them. A player,
betting on $X_{i}$, wins when $X_{i}=X_{1:n}$, namely it is convenient to
bet on $X_{i\text{ }}$ when $X_{i\text{ }}$ is small w.r.t. $\mathbf{X}%
$.$_{\text{ }}$In such a context, the pair-determined property guarantees that
the choice of betting on $X_{i}$ is justified all the times that only elements
$X_{j\text{ }}$with $j\in V_{[i]}$ take part in the competition.

A simple case when  all the variables are pair-determined is given in the next
Example \ref{exem}. A case where not all the  variables are pair-determined can, on the contrary,
be found in Example~\ref{no2}.

\begin{example} \label{exem}
Consider a triple $X_{1},X_{2},X_{3}$ such that
\[
X_{1}\preceq_{sp}X_{2},X_{2}\preceq_{sp}X_{3},X_{3}\preceq_{sp}X_{1}%
\]
and the inequalities are understood in a  ``strict" sense.  Thus
we have that each single variable is trivially pair-determined since we have%
\[
V_{\left[  1\right]  }=\{2\},V_{\left[  2\right]  }=\{3\},V_{\left[  3\right]
}=\{1\}.
\]
\end{example}

Reminding the definition, given above,  of the symbol $ \preceq_{sp}^{[A]}$, we now present the following

\begin{definition}
\label{diFabioBuda} The vector $\mathbf{X} =(X_1, \ldots , X_n )$ is \emph{ordered by
pairs} when $X_i \preceq_{sp} X_j $ implies that
$$
X_i \preceq_{sp}^{[A]} X_j
$$
for any $A \subset [n]$ and $i,j \in A$.
\end{definition}

The ordered by pairs property is indeed rather strong and has a number of implications as shown next.

\begin{prop} \label{consegue}
If  $\mathbf{X} =(X_1, \ldots , X_n )$ is ordered by
pairs then the following properties hold:
\begin{itemize}
\item[a)] For any $i \in [n]$, $X_i$ has the pair-determined property;
\item[b)] If $X_i \preceq_{sp} X_j $ and  $X_j \preceq_{sp} X_k $ then  $X_i \preceq_{sp}^{[A]} X_k $, for any $A \subset [n]$ such that $i,k \in A$;
\item[c)] $X_1 \preceq_{sp} X_j $, for $j= 2, \ldots , n $ if and only if $X_1$ is weakly small w.r.t. $\mathbf{X}_{A } $ for any $A \subset [n]$ such that $1 \in A$.
\end{itemize}
\end{prop}
\begin{proof}
 a) To fix ideas we assume that the variables are indexed in such a way that
$$
\alpha_1 \geq \alpha_2 \geq \cdots \geq \alpha_n  .
$$
Namely,  for $i < j $, one has $X_i \preceq_{sp}^{[n]} X_j$.  By taking into account the ordered by pairs property
for $\mathbf{X} $, one has $X_i \preceq_{sp} X_j$.
Thus,
$$
V_{[i]} =\{  i+1, \ldots , n         \}
$$
for $i = 1, \ldots , n-1 $ and $V_{[n]} = \emptyset$.
Fix now $A$ such that $i \in A$ and $A \subset V_{[i]} \cup \{i \}$. By applying again the ordered by pairs property we obtain
$$
\alpha_i^{[A]}    \geq         \alpha_k^{[A]} ,
$$
for any $k \in A$. Thus $X_i $ is weakly small in $ X_{A \cup \{ i\}}$. Whence we can conclude that $\mathbf{X}$ has the pair--determined property.

b) By hypothesis, $X_i \preceq_{sp} X_j $ and $X_j \preceq_{sp} X_k$. Then the property of ordered by pairs yields
$$
X_i \preceq^{[B]}_{sp} X_j, X_j \preceq^{[B]}_{sp} X_k
$$
where $B = \{i,j,k\}$. Namely, $\alpha_i^{[B]} \geq \alpha_j^{[B]} \geq    \alpha_k^{[B]}        $. Then, by applying again the property of ordered by pairs we obtain
that for any $A $ with $i,k \in A$ one has $ X_i \preceq^{[A]}_{sp} X_k$.

The proof of c) is similar to the above and it can be omitted.
\end{proof}

As an immediate consequence of Proposition \ref{consegue}, one obtains that both non--transitivity and aggregation/marginalization paradoxes can be excluded under the ordered by pairs property. Thus Example \ref{exem} shows a case where the latter property fails
even if all the variables are pair determined.

\medskip

\begin{rem}
Concerning  Barlow-Proschan importance indexes of components in a coherent system, we can
cite a simple application of Proposition \ref{consegue}. Let the vector $\left(
X_{1},...,X_{n}\right)  $ of components' lifetimes have the ordered by pairs
property and compare two components $i,j$. Then the condition
$X_{i}\preceq_{sp}X_{j}$ implies that the component $i$ has a greater importance index, within any path set, than the component $j$.  See also Example \ref{esempioSpi}.
\end{rem}

The following results give some sufficient conditions for the pair--determined, ordered by pairs, or weakly small properties.

\begin{lemma}\label{3rv}
Let $Y_{1},Y_{2},Z$ be independent random variables with $Y_{1}%
\preceq_{st}Y_{2}$. Then%
\[
\mathbb{P}\left(  Y_{1}\leq\min\left(  Y_{2},Z\right)  \right)
\geq \mathbb{P}\left(  Y_{2}\leq\min\left(  Y_{1},Z\right)  \right)  .
\]
\end{lemma}

\begin{proof}

Denote by $f_{Y_{i}}$ the marginal density function of $Y_{i}$, for $i=1,2$ and by $f_Z$ the one of $Z$.

\begin{equation*}
\mathbb{P}\left(  Y_{1}\leq\min\left(  Y_{2},Z\right)  \right)  =\int%
_{0}^{+\infty}\mathbb{P}\left(  Y_{1}\leq\min\left(  Y_{2},Z\right)
|Z=\xi\right)  f_{Z}\left(  \xi\right)  d\xi=
\end{equation*}
\begin{equation}\label{form1}
=\int_{0}^{+\infty}\mathbb{P}\left(  Y_{1}\leq\min\left(  Y_{2},\xi\right)
\right)  f_{Z}\left(  \xi\right)  d\xi,
\end{equation}
where the second identity follows by the assumption of stochastic independence.

Similarly

\begin{equation}\label{form2}
\mathbb{P}\left(  Y_{2}\leq\min\left(  Y_{1},Z\right)  \right)  =\int%
_{0}^{+\infty}\mathbb{P}\left(  Y_{2}\leq\min\left(  Y_{1},\xi\right)
\right)  f_{Z}\left(  \xi\right)  d\xi.
\end{equation}

Now, for any $\xi>0$, we can write%
\begin{equation}\label{form3}
\mathbb{P}\left(  Y_{1}\leq\min\left(  Y_{2},\xi\right)  \right)  =\int%
_{0}^{+\infty}f_{Y_{2}}\left(  y\right)  \left[  \int_{0}^{\min\left(
y,\xi\right)  }f_{Y_{1}}\left(  x\right)  dx\right]  dy,
\end{equation}
and%

\begin{equation}\label{form4}
\mathbb{P}\left( Y_{2}\leq\min\left(  Y_{1},\xi\right)  \right)  =\int%
_{0}^{+\infty}f_{Y_{1}}\left(  x\right)  \left[  \int_{0}^{\min\left(
x,\xi\right)  }f_{Y_{2}}\left(  y\right)  dy\right]  dx.
\end{equation}
For any $\xi>0$, the functions
\[
\rho_{\xi}(u):=\int_{0}^{\min\left(  u,\xi\right)  }f_{Y_{1}}\left(  x\right)
dx,\,\,\, \sigma_{\xi}(u):=\int_{0}^{\min\left(  u,\xi\right)  }f_{Y_{2}}\left(
x\right)  dx
\]
are non-decreasing function w.r.t. $u>0$, and
\[
\rho_{\xi}(u)\geq\sigma_{\xi}(u)
\]
in view of the assumption $Y_{1}\preceq_{st}Y_{2}$. The same assumption then
puts us in a position to conclude%
\[
\mathbb{P}\left(  Y_{1}\leq\min\left(  Y_{2},\xi\right)  \right)  =\int%
_{0}^{+\infty}f_{Y_{2}}\left(  u\right)    \rho_{\xi}(u)  du\geq
\]

\[
\geq\int_{0}^{+\infty}f_{Y_{2}}\left(  u\right)    \sigma_{\xi
}(u)  du\geq\int_{0}^{+\infty}f_{Y_{1}}\left(  u\right)
\sigma_{\xi}(u)  du=\mathbb{P}\left(  Y_{1}\leq\min\left(  Y_{2}%
,\xi\right)  \right)  .
\]
Whence the thesis is obtained by recalling equations \eqref{form1}, \eqref{form2} and by integrating the functions in \eqref{form3} and \eqref{form4}  with respect to the variable
$\xi$.
\end{proof}

We can now obtain
 a simple sufficient
condition, in the case of independent random variables, ensuring that a single
random variable $X_{1}$ is simultaneously pair-determined and weakly small in
$\mathbf{X}$.

\begin{prop}
\label{fo} Let $\mathbf{X} =(X_1, \ldots , X_n )$ be a vector of independent random variables such that $X_1 \preceq_{st.} X_i$, for $i =2,
\ldots, n$. Then $X_1$ is pair-determined in $\mathbf{X}$ and it is weakly
small w.r.t. $\mathbf{X} $. Moreover, if the random variables are
not identically distributed then $X_1$ is small w.r.t. $\mathbf{X} $.
\end{prop}
\begin{proof}
Fix $j$ and  a set $A \subset [n] $ such that $1,j \in A$. We recall the notation   $X_{1: A \setminus \{1, j\}}= \min_{\ell \in A , \ell \neq 1,j  }  X_\ell$.
From the assumption that $X_1 \preceq_{st} X_j $ and from the above Lemma \ref{3rv}, we immediately get
$$
\alpha_1^{[A]} = \mP ( X_1 \leq X_j \wedge  X_{1: A \setminus \{1, j\}} ) \geq \mP ( X_j \leq X_1 \wedge   X_{1: A \setminus \{1, j\}}) \geq \alpha_j^{[A]} .
$$
The proof can be concluded by recalling Definition \ref{couple}.
\end{proof}

\begin{prop}
\label{fo2} Let $\mathbf{X} =(X_1, \ldots , X_n )$ be a vector of independent random variables such that $X_{i-1} \preceq_{st.} X_{i }$, for $i =2,
\ldots, n$. Then $\mathbf{X} $ is ordered by pairs.
\end{prop}
\begin{proof}
Fix $i <j$ and  a set $A \subset [n] $ such that $i,j \in A$.
From the assumption that $X_i \preceq_{st} X_j $ and from the above Lemma \ref{3rv}, we immediately get
$$
\alpha_i^{[A]} = \mP ( X_i \leq X_j \wedge  X_{1: A \setminus \{i, j\}} ) \geq \mP ( X_j \leq X_i \wedge  X_{1: A \setminus \{i, j\}} ) \geq \alpha_j^{[A]} .
$$
The proof can be concluded by recalling Definition \ref{diFabioBuda}.
\end{proof}

We now pass to consider the case of non-independent random variables and focus attention on the family of
the m.c.h.r.'s. First of all we have
the following simple conclusion.
\begin{prop}\label{Budapest}
If $\lambda_1 (t | \emptyset ) \geq \lambda_j (t | \emptyset )$, for any $t \geq 0 $ and $j \geq 2$, then
$X_1$ is weakly small w.r.t. $    \mathbf{X} =    (X_1, \ldots , X_n)$.
\end{prop}
\begin{proof}
It is an immediate consequence of formula \eqref{Rome}.
\end{proof}
In our analysis a simplifying condition is the one of ``initial-time-homogeneity".
We say that $\mathbf{X} =(X_1, \ldots , X_n)$ is \emph{initially time homogeneous} if there exist  constants $(\beta_j : j =1, \ldots , n )$ such that $ \lambda_j (t | \emptyset )  = \beta_j  $.

\begin{rem}\label{concludiamo?}
Let us focus attention on the independent random variables $Z_1, \ldots , Z_n $ introduced in Proposition \ref{prop-ind} and considered in equation \eqref{Exp(c)Gen}  and \eqref{Rome}.  If 
$\mathbf{X} =(X_1, \ldots , X_n )$ is  initially time-homogeneous then $Z_1, \ldots , Z_n $ are exponentially distributed.  
\end{rem} 

In view of the above remark,  
an obvious corollary of Proposition \ref{Budapest} is the following one
\begin{corollary} \label{coro1}
If   $\mathbf{X} =(X_1, \ldots , X_n)$ is initially time homogeneous,  then
$X_j $ is weakly small w.r.t.  $  \mathbf{X}     $ if and only if $\beta_j = \max_{i =1, \ldots , n } \beta_i$;
\end{corollary}


\medskip

The following simple result shows a sufficient condition for the property of being ordered by pairs.

\begin{prop} \label{sufficient} Assume that $\mathbf{X} $ is initially-time-homogeneous,
with $ \beta_i \neq  \beta_j$ for all $ i \neq j $, 
 and that  the condition $ \beta_\ell >\beta_j   $  implies
$ \lambda_\ell (t|I; t_1, \ldots , t_{|I|})\geq   \lambda_j(t|I; t_1, \ldots , t_{|I|})    $, for any $I \subset [n]$,  $t_1, \ldots , t_{|I|}$ and $t > t_{|I|}$. Then the condition of ordered by pairs holds true.
\end{prop}
\begin{proof} Without any loss of generality we can consider that
\begin{equation}\label{beta1}
\beta_1 >\beta_2 >  \cdots >\beta_n .
\end{equation}
 In order to obtain the thesis
we must prove that, for any $ j < \ell$ and any $A \subset [n]$ such that $  \ell ,j \in A $, $X_j \preceq^{[A]} X_\ell $. We then consider the ``marginal" m.c.h.r function
$$
\lambda_i^{[A]}(t| \emptyset) := \lim_{\Delta t \to 0^+} \frac{\mP(X_i \in (t, t +\Delta t)  | X_{1:A} >t     )}{\Delta t} .
$$

By taking into account the definition of the $\lambda_i (t| I, t_1, \ldots , t_{|I|})$ we obtain
\begin{equation}\label{intA}
\lambda_i^{[A]}(t| \emptyset )= \sum_{I \subset A^c} \int_{[0,t]^{|I|}} \lambda_i(t; I, t_1, \ldots, t_{|I|}) f_{\mathbf{X}_I} (t_1, \ldots, t_{|I|}| X_{1:A} >t)dt_1 \ldots dt_{|I|} .
\end{equation}
By \eqref{intA} and  \eqref{beta1} we obtain that
$$
\lambda_j^{[A]}(t| \emptyset) \geq  \lambda_\ell^{[A]}(t| \emptyset)
$$
where $j, \ell \in A$ with $j < \ell$. Then by formula \eqref{Rome} follows the thesis.
\end{proof}

As a special case of initially-time-homogeneous models, we find the time-homogeneous load-sharing models mentioned in Section 2.
 Even if such a condition is very restrictive, this class of models is relevant in that it can still be seen as a generalization of the condition of independence and exponentiality. It can be interesting to specialize to these cases the preceding results about initially-time-homogeneous models.

\medskip \medskip

\textbf{Acknowledgements.}
We would  like to thank an anonymous Referee for  valuable comments and suggestions which, in
  particular, led us to add Remarks 2, 3, and 6.
Most of the results had been presented
 at the IWAP conference held in Budapest (Hungary), June
2018. E.D.S. and F.S. acknowledge partial support of Ateneo Sapienza Research Projects
\textquotedblleft\textit{Dipendenza,}
\textit{disuguaglianze e approssimazioni in modelli stocastici}" (2015),
``\textit{Processi stocastici: Teoria e applicazioni}" (2016), and
\textquotedblleft\textit{Simmetrie e Disuguaglianze in Modelli Stocastici}"
(2018). Y.M. would like to express his gratitude to coathors for their invitation and support during his visit at Department of Mathematics, Sapienza University of Rome,  in January 2018.

\bibliographystyle{abbrv}

\end{document}